\documentclass{article}

\usepackage{amsmath,amssymb, amsthm}

\newtheorem{theorem}{Theorem}

\newtheorem{corollary}[theorem]{Corollary}

\newtheorem{remark}[theorem]{Remark}

\begin{document}

\title{\textbf{Estimates of convolution operators of functions from $L_{2\pi}^{p(x)}$ }}
\author{\textbf{W\l odzimierz \L enski \ and Bogdan Szal} \\
University of Zielona G\'{o}ra\\
Faculty of Mathematics, Computer Science and Econometrics\\
65-516 Zielona G\'{o}ra, ul. Szafrana 4a, Poland\\
W.Lenski@wmie.uz.zgora.pl, B.Szal @wmie.uz.zgora.pl}
\date{}
\maketitle

\begin{abstract}
We generalize and slight improve the result of I. I. Sharapudinov [Mat.
Zametki, 1996, Volume 59, Issue 2, 291--302]. Some applications to the de la
Vall\'{e}e Poussin operator will also be given.
\end{abstract}

\ \ \ \ \ \ \ \ \ \ \ \ \ \ \ \ \ \ \ \ 


\ \ \ \ \ \ \ \ \ \ \ \ \ \ \ \ \ \ \ 


\section{Introduction}

Let $p=p(x)$ will be a measurable $2\pi $ - periodic function, $p_{-}=\inf
\left\{ p(x):x\in 
\mathbb{R}
\right\} $, $p^{-}$ $=\sup \left\{ p(x):x\in 
\mathbb{R}
\right\} $, $1$ $\leq p_{-}\leq $ $p^{-}$ $<\infty $ and $L_{2\pi }^{p}$
will be the space of measurable $2\pi $ - periodic functions $f$ such that $%
\int_{-\pi }^{\pi }\left\vert f(x)\right\vert ^{p(x)}dx<\infty $.

Putting%
\[
\left\Vert f\right\Vert _{p}=\inf \left\{ \alpha >0:\int_{-\pi }^{\pi
}\left\vert \frac{f(x)}{\alpha }\right\vert ^{p(x)}dx\leq 1\right\} 
\]%
we turn $L_{2\pi }^{p}$ into the Banach space. We write $\Pi _{2\pi }$ for
the set of all $2\pi $ - periodic variable exponents $p=p(x)\geq 1$
satisfying the condition%
\begin{equation}
\left\vert p(x)-p(y)\right\vert \ln \frac{2\pi }{\left\vert x-y\right\vert }%
=O\left( 1\right) \text{,\ }\left( x,y\in \left[ -\pi ,\pi \right] \right) .
\label{1}
\end{equation}

In the paper \cite{IIS1} I. I. Sharapudinov proved the following theorem:

\noindent \textbf{Theorem A.} \textit{Let }$k_{\lambda }=k_{\lambda }(x)$\textit{\ }$%
\left( 1\leq \lambda <\infty \right) $\textit{\ be a measurable }$2\pi $%
\textit{\ -periodic essentially bounded function (kernel) satisfying the
conditions:}

\noindent \textit{A}$^{\circ }$\textit{) }$\int_{-\pi }^{\pi }\left\vert k_{\lambda
}(t)\right\vert dt\leq c_{1}^{\circ },$

\noindent \textit{B}$^{\circ }$\textit{) }$\sup_{t}\left\vert k_{\lambda
}(t)\right\vert \leq c_{2}^{\circ }\lambda ^{\eta }$\textit{,}

\noindent \textit{C}$^{\circ }$\textit{) }$\left\vert k_{\lambda }(t)\right\vert \leq
c_{3}^{\circ }$\textit{, \ }$\left( \lambda ^{-\gamma }\leq \left\vert
t\right\vert <\pi \right), $

\noindent \textit{where }$\gamma ,\eta ,$\textit{\ }$c_{1}^{\circ },c_{2}^{\circ
},c_{3}^{\circ }>0$\textit{\ are independent of }$\lambda $\textit{.}

\textit{If }$f\in L_{2\pi }^{p}$\textit{\ with }$p=p(x)$\textit{\ }$\in \Pi
_{2\pi }$\textit{, then}%
\[
\left\Vert K_{\lambda }[f]\right\Vert _{p}=O\left( 1\right) \left\Vert
f\right\Vert _{p}, 
\]%
\textit{where}%
\[
K_{\lambda }[f](x)=\int_{-\pi }^{\pi }f(t)k_{\lambda }(t-x)dt. 
\]%
We generalize and slight improve this result considering the wider family of
two parameters convolution operators. Some applications to the de la Vall%
\'{e}e Poussin operator will also be given.

\section{Main result}

Denote by $k_{m,n}=k_{m,n}(x),$ for every $0\leq m\leq n<\infty,$ a measurable $%
2\pi $ -periodic essentially bounded function (kernel). Let define the
linear operator

\[
K_{m,n}[f]=K_{m,n}[f](x)=\int_{-\pi }^{\pi }f(t)k_{m,n}(t-x)dt 
\]%
in space $L_{2\pi }^{p}.$ We will say that the kernel family $\left\{
k_{m,n}(x)\right\} _{0\leq m\leq n<\infty }$ satisfies the conditions B) and
C), respectively, if the following estimates hold:

\noindent B) $\sup_{\left\vert t\right\vert \leq h_{m}}\left\vert
k_{m,n}(t)\right\vert \leq c_{2}\left( n+1\right) ^{\eta }$,

\noindent C) $\left\vert k_{m,n}(t)\right\vert \leq c_{3}$, \ $\left( h_{m}\leq
\left\vert t\right\vert <\pi \right) ,$

\noindent where $h_{m}=\frac{\pi }{\left( m+1\right) ^{\gamma }}$ and $\gamma$, $\eta
,c_{2},c_{3}>0$ are independent of $m,n$.

For the operator $K_{m,n}[f]$ we will prove the following general estimate:

\begin{theorem}
Let $k_{m,n}=k_{m,n}(x)$ $\left( 0\leq m\leq n<\infty \right) $ satisfy the
conditions B) and C). If $f\in L_{2\pi }^{p}$ with $p=p(x)$ $\in \Pi _{2\pi
} $, then%
\[
\left\Vert K_{m,n}[f]\right\Vert _{p}=O\left( \mu _{m,n}\right) \left\Vert
f\right\Vert _{p}, 
\]%
where 
\[
\mu _{m,n}=\mu \left( k_{m,n}\right) =\int_{-h_{m}}^{h_{m}}\left\vert
k_{m,n}(t)\right\vert dt. 
\]
\end{theorem}

\begin{proof}
Let%
\begin{equation}
x_{k}=kh_{n}-\pi ,\text{ \ \ \ \ \ \ \ \ \ \ \ \ \ \ \ \ \ \ \ \ \ \ \ \ \ \
\ \ \ \ \ \ \ \ }k=0,\pm 1,\pm 2,...  \label{2}
\end{equation}%
\begin{equation}
s_{k}=\min \left\{ p\left( x\right) :x_{k-1}\leq x\leq x_{k+2}\right\} ,%
\text{ }k=0,\pm 1,\pm 2,...  \label{3}
\end{equation}%
\begin{equation}
p_{t}\left( x\right) =s_{k}\text{ \ \ }\left( x_{k-1}-t\leq x\leq
x_{k+1}-t\right) ,\text{ \ }k=0,\pm 1,\pm 2,...  \label{4}
\end{equation}%
whence $p_{t}\left( x\right) =p_{0}\left( x+t\right) $ is a $2\pi $ -
periodic step function such that%
\[
p_{t}\left( x\right) \leq p\left( x\right) . 
\]

Denote by 
\[
E_{m}\left( x\right) =\left( -\pi ,\pi \right) \backslash \left(
x-h_{m},x+h_{m}\right) , 
\]%
when $\left( x-h_{m},x+h_{m}\right) \subset \left( -\pi ,\pi \right) $, but 
\[
E_{m}\left( x\right) =\left( -\pi ,\pi \right) \backslash \left( -\pi
,x+h_{m}\right) 
\]%
or 
\[
E_{m}\left( x\right) =\left( -\pi ,\pi \right) \backslash \left( x-h_{m},\pi
\right) , 
\]%
when $x-h_{m}<-\pi $\ or $\pi <x+h_{m}$, respectively.

Let%
\[
\left\Vert f\right\Vert _{p}\leq 1.
\]%
It is clear that for $\overline{p}=\max \left\{ p\left( x\right) :-\pi \leq
x\leq \pi \right\} $%
\begin{eqnarray*}
&&\left( \int_{-\pi }^{\pi }\left\vert K_{m,n}[f](x)\right\vert
^{p(x)}dx\right) ^{1/\overline{p}} \\
&=&\left( \int_{-\pi }^{\pi }\left\vert \int_{-\pi }^{\pi
}f(t)k_{m,n}(t-x)dt\right\vert ^{p(x)}dx\right) ^{1/\overline{p}} \\
&=&\left( \int_{-\pi }^{\pi }\left\vert \left( \int_{E_{m}\left( x\right)
}+\int_{x-h_{m}}^{x+h_{m}}\right) f(t)k_{m,n}(t-x)dt\right\vert
^{p(x)}dx\right) ^{1/\overline{p}}
\end{eqnarray*}%
\begin{eqnarray*}
&\leq &\left( \int_{-\pi }^{\pi }\left\vert \int_{E_{m}\left( x\right)
}f(t)k_{m,n}(t-x)dt\right\vert ^{p(x)}dx\right) ^{1/\overline{p}} \\
&&+\left( \int_{-\pi }^{\pi }\left\vert
\int_{x-h_{m}}^{x+h_{m}}f(t)k_{m,n}(t-x)dt\right\vert ^{p(x)}dx\right) ^{1/%
\overline{p}} \\
&=&J_{m}^{1/\overline{p}}+J_{x}^{1/\overline{p}}.
\end{eqnarray*}%
Since%
\[
\left\Vert f\right\Vert _{q}\leq \left( 2\pi +1\right) \left\Vert
f\right\Vert _{p}
\]%
with $q\left( x\right) \leq p\left( x\right) $ for $-\pi \leq x\leq \pi $
(cf. \cite{IIS0}), by condition C), 
\begin{eqnarray*}
&&\left\vert \int_{E_{m}\left( x\right) }f(t)k_{m,n}(t-x)dt\right\vert  \\
&=&O\left( 1\right) \int_{E_{m}\left( x\right) }\left\vert f(t)\right\vert
dt=O\left( 1\right) \int_{-\pi }^{\pi }\left\vert f(t)\right\vert dt=O\left(
1\right) \left\Vert f\right\Vert _{p}=O\left( 1\right) 
\end{eqnarray*}%
and therefore 
\[
J_{m}=O\left( 1\right) .
\]%
In case of integral $J_{x}$, using (\ref{1}), (\ref{2}) and (\ref{3}), for $%
x_{k}\leq x\leq x_{k+1},$ we obtain that%
\[
\left\vert p(x)-s_{k}\right\vert =O\left( \frac{1}{\ln 2\left( n+1\right)
^{\gamma }}\right) ,
\]%
and therefore by condition B)%
\begin{eqnarray*}
\left\vert \int_{x-h_{m}}^{x+h_{m}}f(t)k_{m,n}(t-x)dt\right\vert
^{p(x)-s_{k}} &=&\left( n+1\right) ^{\eta O\left( \frac{1}{\ln 2\left(
n+1\right) ^{\gamma }}\right) }\left( \int_{x-h_{m}}^{x+h_{m}}\left\vert
f(t)\right\vert dt\right) ^{O\left( \frac{1}{\ln 2\left( n+1\right) ^{\gamma
}}\right) } \\
&=&O\left( 1\right) \left( \left\Vert f\right\Vert _{p}\right) ^{O\left( 
\frac{1}{\ln 2\left( n+1\right) ^{\gamma }}\right) }=O\left( 1\right) .
\end{eqnarray*}%
Next,%
\begin{eqnarray*}
J_{x} &=&\sum_{k=0}^{2[n^{\gamma }]}\int_{x_{k}}^{x_{k+1}}\left\vert
\int_{x-h_{m}}^{x+h_{m}}f(t)k_{m,n}(t-x)dt\right\vert ^{s_{k}+p(x)-s_{k}}dx
\\
&=&\sum_{k=0}^{2[n^{\gamma }]}\int_{x_{k}}^{x_{k+1}}\left\vert
\int_{x-h_{m}}^{x+h_{m}}f(t)k_{m,n}(t-x)dt\right\vert
^{p(x)-s_{k}}\left\vert
\int_{x-h_{m}}^{x+h_{m}}f(t)k_{m,n}(t-x)dt\right\vert ^{s_{k}}dx.
\end{eqnarray*}%
Thus 
\[
J_{x}=O\left( 1\right) \sum_{k=0}^{2[n^{\gamma }]}\left( \mu _{m,n}\right)
^{s_{k}}\int_{x_{k}}^{x_{k+1}}\left\vert \frac{1}{ \mu _{m,n}}\int_{x-h_{m}}^{x+h_{m}}f(t)k_{m,n}(t-x)dt\right\vert ^{s_{k}}dx.
\]%
By the Jensen inequality,%
\begin{eqnarray*}
J_{x} &=&O\left( 1\right) \sum_{k=0}^{2[n^{\gamma }]}\left( \mu
_{m,n}\right) ^{s_{k}}\int_{x_{k}}^{x_{k+1}}\left\vert \frac{1}{\mu _{m,n}}%
\int_{x-h_{m}}^{x+h_{m}}f(t)k_{m,n}(t-x)dt\right\vert ^{s_{k}}dx \\
&=&O\left( 1\right) \sum_{k=0}^{2[n^{\gamma }]}\left( \mu _{m,n}\right)
^{s_{k}-1}\int_{x_{k}}^{x_{k+1}}\left[ \int_{x-h_{m}}^{x+h_{m}}\left\vert
f(t)\right\vert ^{s_{k}}\left\vert k_{m,n}(t-x)\right\vert dt\right] dx \\
&=&O\left( 1\right) \left( \mu _{m,n}\right) ^{\overline{p}%
-1}\int_{-h_{m}}^{h_{m}}\left\vert k_{m,n}(t)\right\vert \left(
\sum_{k=0}^{2[n^{\gamma }]}\int_{x_{k}-t}^{x_{k+1}-t}\left\vert
f(x)\right\vert ^{s_{k}}dx\right) dt \\
&=&O\left( 1\right) \left( \mu _{m,n}\right) ^{\overline{p}%
-1}\int_{-h_{m}}^{h_{m}}\left\vert k_{m,n}(t)\right\vert \left( \int_{-\pi
-t}^{\pi -t}\left\vert f(x)\right\vert ^{p_{t}\left( x\right) }dx\right) dt
\\
&=&O\left( 1\right) \left( \mu _{m,n}\right) ^{\overline{p}%
-1}\int_{-h_{m}}^{h_{m}}\left\vert k_{m,n}(t)\right\vert \left( \int_{-\pi
}^{\pi }\left\vert f(x)\right\vert ^{p_{t}\left( x\right) }dx\right) dt.
\end{eqnarray*}%
Similar to \cite[(17) p.295]{IIS1} we have%
\[
\int_{-\pi }^{\pi }\left\vert f(x)\right\vert ^{p_{t}\left( x\right) }dx\leq
\left( 2\pi +1\right) ^{\overline{p}}.
\]%
Hence%
\[
J_{x}^{1/\overline{p}}=O\left( \mu _{m,n}\right) 
\]%
and our result follows.
\end{proof}

\begin{corollary}
If we put $m=n=\lambda $ in the assumptions of Theorem 1, then for $f\in
L_{2\pi }^{p}$ with $p=p(x)$ $\in \Pi _{2\pi }$ the following estimate%
\[
\left\Vert K_{\lambda }[f]\right\Vert _{p}=O\left( \mu _{\lambda }\right)
\left\Vert f\right\Vert _{p} 
\]%
holds, \textit{where }$\mu _{\lambda }=\int_{-h_{\lambda }}^{h_{\lambda
}}\left\vert k_{\lambda }(t)\right\vert dt$ and $k_{\lambda }=k_{\lambda, \lambda }$.
\end{corollary}

\begin{remark}
If we additionally assume that $\mu _{\lambda }=O\left( 1\right) ,$\ then
Corollary 2 gives\ Theorem A with the result \cite{IIS1} of I. I.
Sharapudinov, under the slight weaker conditions.
\end{remark}

\section{De la Vall\'{e}e Poussin operator}

Let $f\in L_{2\pi }^{1}$ and consider the trigonometric Fourier series 
\[
Sf(x):=\frac{a_{0}(f)}{2}+\sum_{\nu =1}^{\infty }(a_{\nu }(f)\cos \nu
x+b_{\nu }(f)\sin \nu x) 
\]%
with the partial sums\ $S_{k}f.$ For $0\leq m\leq n,$ $m, n=0,1,2,...$\ denote
by%
\[
V_{n,m}[f]\left( x\right) =\frac{1}{m+1}\sum_{k=n-m}^{n}S_{k}[f]\left(
x\right) =\frac{1}{\pi \left( m+1\right) }\int_{-\pi }^{\pi }f\left(
t\right) \Phi _{n,m}\left( t-x\right) dt 
\]%
the de la Vall\'{e}e Poussin means of the series $Sf,$ where%
\[
\Phi _{n,m}\left( t\right) =\frac{1}{m+1}\frac{\sin \frac{\left( m+1\right) t%
}{2}\sin \frac{\left( 2n-m+1\right) t}{2}}{2\sin ^{2}\frac{t}{2}}. 
\]

It is clear that the kernel family $\left\{ \Phi _{m,n}(x)\right\} _{1\leq
m\leq n<\infty }$ satisfies the conditions B) with $\eta =1$ and C) with $%
\gamma =\frac{1}{2}$. By the following calculation%
\begin{eqnarray*}
&&\frac{1}{m+1}\int_{0}^{h_{m}}\frac{\left\vert \sin \frac{\left( m+1\right)
t}{2}\sin \frac{\left( 2n-m+1\right) t}{2}\right\vert }{2\sin ^{2}\frac{t}{2}%
}dt \\
&=&\frac{1}{m+1}\left( \int_{0}^{\frac{\pi }{n+1}}+\int_{\frac{\pi }{n+1}}^{%
\frac{\pi }{m+1}}+\int_{\frac{\pi }{m+1}}^{h_{m}}\right) \frac{\left\vert
\sin \frac{\left( m+1\right) t}{2}\sin \frac{\left( 2n-m+1\right) t}{2}%
\right\vert }{2\sin ^{2}\frac{t}{2}}dt \\
&\leq &\int_{0}^{\frac{\pi }{n+1}}\frac{\left( 2n-m+1\right) dt}{2}+\int_{%
\frac{\pi }{n+1}}^{\frac{\pi }{m+1}}\frac{dt}{2\frac{t}{2}\frac{2}{\pi }}+%
\frac{1}{m+1}\int_{\frac{\pi }{m+1}}^{h_{m}}\frac{dt}{2\left( \frac{t}{2}%
\frac{2}{\pi }\right) ^{2}} \\
&\leq &\pi +\frac{\pi }{2}\ln \frac{n+1}{m+1}+\frac{\pi }{2}
\end{eqnarray*}%
we obtain%
\[
\mu \left( \Phi _{m,n}\right) =\int_{-h_{m}}^{h_{m}}\left\vert \Phi
_{m,n}(t)\right\vert dt\leq \pi \left( 3+\ln \frac{n+1}{m+1}\right) . 
\]%
Hence, by Theorem 1, we have:

\begin{theorem}
If $L_{2\pi }^{p}$ with $p=p(x)$ $\in \Pi _{2\pi }$, then%
\[
\left\Vert V_{n,m}[f]\right\Vert _{p}=O\left( 1+\ln \frac{n+1}{m+1}\right)
\left\Vert f\right\Vert _{p}\text{ \ \ \ }\left( 0\leq m\leq n,\text{ }%
m, n=0,1,2,...\right) . 
\]
\end{theorem}
From Theorem 4 we get the following corollary:
\begin{corollary}
Let $f\in L_{2\pi }^{p}$ with $p=p(x)$ $\in \Pi _{2\pi }.$ If $m=O\left( n\right) $, then%
\[
\left\Vert V_{n,m}[f]\right\Vert _{p}=O\left( 1\right) \left\Vert
f\right\Vert _{p} 
\]%
hold.
\end{corollary}

In the special case we can consider the following Fourier operator:%
\[
S_{n}[f]\left( x\right) =V_{n,0}[f]\left( x\right) =\frac{1}{\pi }\int_{-\pi
}^{\pi }f\left( t\right) \Phi _{n,0}\left( t-x\right) dt, 
\]%
where%
\[
\Phi _{n,0}\left( t\right) =\frac{\sin \frac{\left( 2n+1\right) t}{2}}{2\sin 
\frac{t}{2}}. 
\]%
For this operator we have:

\begin{corollary}
Let $f\in L_{2\pi }^{p}$ with $p=p(x)$ $\in \Pi _{2\pi }.$ If we put $m=0$
in Theorem 4, then 
\[
\left\Vert S_{n}[f]\right\Vert _{p}=O\left( 1+\ln \left( n+1\right) \right)
\left\Vert f\right\Vert _{p}\text{ \ \ \ }\left( \text{ }n=0,1,2,...\right)
. 
\]%
\end{corollary}

\begin{remark}
In the case $p\equiv 1$ the results of this section we can find in the
monograph of A. Zygmund \cite[Ch. II, p.70, Ch. III, p.90]{AZ} (see e.g.\cite%
[Ch. II, p.117-8]{WZ} .
\end{remark}


\begin{thebibliography}{9}
\bibitem{IIS0} Sharapudinov, I. I.: The basis property of the Haar system in
the space $L^{p(x)}([0;1])$ and the principle of localization in the mean,
Math. Sb., Vol. 130(172), No 2(6). 275-283 (1986).

\bibitem{IIS1} Sharapudinov, I. I.: Uniform boundedness in $L^{p}(p=p(x))$ of
some families of convolution operators, Mat. Zametki, Volume 59, Issue
2, 291-302 (1996).

\bibitem{IIS2} Sharapudinov, I. I.: On direct and inverse theorems of
approximation theory in variable Lebesgue and Sobolev spaces, Azerbaijan
Journal of Mathematics V. 4, No 1, January (2014).

\bibitem{IIS3} Sharapudinov, I. I.: Approximation of functions in $L_{2\pi
}^{p(x)}$ by trigonometric polynomials, Izv. RAN. Ser. Mat., Volume
77, Issue\ 2, 197-224 (2013).

\bibitem{AZ} Zygmund, A.: Trigonometric series, Cambridge, London, New York,
Melbourne, (2002).

\bibitem{WZ} Zhuk, V. V.: Approximation of periodic functions (Russian),
Leningrad, (1982).
\end{thebibliography}
\end{document}